\documentclass[11pt,letterpaper]{amsart}

\usepackage{amsmath,amscd}

\usepackage{setspace}
\usepackage{multicol}
\usepackage[lite]{amsrefs} 
\usepackage{amssymb} 
\usepackage{graphicx} 
\usepackage{euscript}
\usepackage{color}
\usepackage{array}
\usepackage{picture}
\usepackage{epic}
\usepackage{tikz}
 \usetikzlibrary{backgrounds,cd}
\usepackage{comment}
\usepackage{enumerate}
\usepackage{url}
\usepackage[colorlinks=true, linkcolor=black, citecolor=black, urlcolor=black]{hyperref}

\usepackage{bookmark}
\usepackage[margin=1.4in]{geometry}
\usepackage{caption}
\usepackage{MnSymbol}
\usepackage{enumitem}
\usepackage{fancyhdr} 
\usepackage{mathtools}

\usepackage{braket}

\usepackage{xcolor}
\usepackage{rotating}
\usepackage[all]{xy}
\usepackage{comment}
\usepackage{pdflscape}
\usepackage{tikz}
\usepackage{tikz-cd}
\usepackage{soul,color}
\usepackage[colorinlistoftodos]{todonotes}

\usepackage{tabmacD}

\newcommand*{\fullref}[1]{\hyperref[{#1}]{\ref*{#1}. \nameref*{#1}}}

\tikzset{
  symbol/.style={
    draw=none,
    every to/.append style={
      edge node={node [sloped, allow upside down, auto=false]{$#1$}}
    }
  }
}

\usepackage{amsfonts}
\usepackage{amssymb}
\usepackage{extarrows}
\usepackage{cleveref}
\usepackage{amsthm}
\usepackage{amscd}
\usepackage{amsmath}
\usepackage{mathtools}
\usepackage{mathrsfs}

\usepackage{color}	
\definecolor{due}{RGB}{0,76,147}

\usepackage{graphicx}	
\usepackage{multicol}	
\usepackage{wrapfig}
\usepackage{enumitem}

\usepackage{longtable}

\theoremstyle{definition}
\newtheorem{defi}{Definition}[section]
\theoremstyle{plain}
\newtheorem{thm}[defi]{Theorem}

\newtheorem{introthm}{Theorem}[section]

\newtheorem{cor}[defi]{Corollary}
\newtheorem{lemma}[defi]{Lemma}
\theoremstyle{remark}

\newtheorem{rmk}[defi]{Remark}

\theoremstyle{definition}

\newtheorem*{ack}{Acknowledgement}

  \makeatletter
\newcommand{\xdashrightarrow}[2][]{\ext@arrow 0359\rightarrowfill@@{#1}{#2}}
\newcommand{\xdashleftarrow}[2][]{\ext@arrow 3095\leftarrowfill@@{#1}{#2}}
\newcommand{\xdashleftrightarrow}[2][]{\ext@arrow 3359\leftrightarrowfill@@{#1}{#2}}
\def\rightarrowfill@@{\arrowfill@@\relax\relbar\rightarrow}
\def\leftarrowfill@@{\arrowfill@@\leftarrow\relbar\relax}
\def\leftrightarrowfill@@{\arrowfill@@\leftarrow\relbar\rightarrow}
\def\arrowfill@@#1#2#3#4{%
  $\m@th\thickmuskip0mu\medmuskip\thickmuskip\thinmuskip\thickmuskip
   \relax#4#1
   \xleaders\hbox{$#4#2$}\hfill
   #3$%
}
\makeatother

\numberwithin{equation}{section}

\makeatletter
\@namedef{subjclassname@2020}{%
  \textup{2020} Mathematics Subject Classification}
\makeatother

\begin{document}
	\title[Nef Cones of Hilbert Schemes of Orthogonal Grassmannians]{Nef Cones of Hilbert Schemes of Orthogonal Grassmannians}
	\author{Minyoung Jeon}
	  \address{Department of Mathematics, Seoul Women’s University, Seoul 01797, Republic of Korea}\email{\url{minyoung.jeon@swu.ac.kr}}

\subjclass[2020]{Primary 14C17, 14C05, 14M15 ; Secondary 14N15, 14C20} %
\keywords{Nef cones, Hilbert schemes, orthogonal Grassmannians}
\date{\today}

\begin{abstract}
We show the number of connected components of the Hilbert scheme of orthogonal Grassmannians under certain condition, and use this result to describe the geometry of the Hilbert scheme. Subsequently, we determine the Nef cone of the Hilbert scheme by identifying curves dual to the generators of its Néron–Severi group. Our results generalize those of ordinary Grassmannians by Seong, and our approach adapts his proof technique alongside relevant Schubert calculus.

\end{abstract}
\maketitle 
\setcounter{tocdepth}{2}
\setcounter{tocdepth}{1}

\section{Introduction}

 The nef cone is a fundamental invariant of a projective variety. It reflects the positivity properties of its divisors and is closely related to ampleness, projective embeddings, and the birational geometry of the variety. Nef cones have been described for various projective varieties, such as those studied in \cites{Miyaoka,BP14}. They have also been studied for several Hilbert schemes, for instance in \cites{BC13,BHLRSWZ}. In particular, the nef cone of the Hilbert scheme of hypersurfaces in a Grassmannian was computed in \cite{Seong}. In this paper, we extend this perspective to orthogonal Grassmannians and determine the generators of the nef cone of the corresponding Hilbert scheme.

It is known from \cite{Hart} that the Hilbert scheme $\mathrm{Hilb}_{P(t)}(X)$ parametrizing subschemes of a projective variety $X$ with Hilbert polynomial $P(t)$ is connected. On the other hand, Seong \cite{Seong20} showed that the Hilbert scheme $\mathrm{Hilb}_{P(T)}(G(k,n))$ of the ordinary Grassmannian $G(k, n)$  need not be connected, where $G(k, n)$ parametrizes $k$-planes in an $n$-space. In this direction we show in Theorem \ref{TheoremA} that the Hilbert scheme $\mathrm{Hilb}_{P_{d,r}(t)}(OG(q,V))$ is not necessarily connected. Here $OG(q,V)$ denotes orthogonal Grassmannians parametrizing the $q$-dimensional isotropic subspaces in an $N$-dimensional space $V$, and $P_{d,r}(t)$ is a polynomial 
\[
P_{d,r}(t):=\binom{t+r}{r}-\binom{t+r-d}{r}.
\]

\begin{introthm}\label{TheoremA}
 Let $m=\mathrm{min}\{\lfloor (N-1)/2 \rfloor-q,q-1\}$.
If $2<r\leq m$, there are $2$ connected components in $\mathrm{Hilb}_{P_{d,r}(t)}(OG(q;V))$. 
In particular, if $m=q-1$, for $q-1<r\leq \lfloor (N-1)/2 \rfloor-q$, there is $1$ connected component in $\mathrm{Hilb}_{P_{d,r}(t)}(OG(q;V))$. 

Consequently, if $d\geq 3$, $2<r\leq \mathrm{min}\{\lfloor(N-1)/2\rfloor-q,q-1\}$, then the Hilbert scheme $\mathrm{Hilb}_{P_{d,r}(t)}(OG(q;V))$ is disconnected.
\end{introthm}

To prove Theorem \ref{TheoremA}, we view $\mathrm{Hilb}_{P_{d,r}(t)}(OG(q,V))$ as degree $d$ hypersurfaces in $r$-dimensional projective spaces of $OG(q,V)$. We then use Schubert Calculus on the orthogonal Grassmannian to establish the disconnectedness. Furthermore, we describe the geometry of the Hilbert scheme $\mathrm{Hilb}_{P_{d,r}(t)}(OG(q,V))$ as a projective bundle over a disjoint union of twoaf isotropic flag varieties. See Theorem \ref{thm3.5}.

Using this geometric description, we provide the generators of the nef cone of the Hilbert scheme, as in Theorem \ref{TheoremB}. Specifically, four generators of the nef cone are obtained as pull-backs of generators of the Nef cones from the two isotropic flag varieties. Two of these generators come from each flag variety. The other two generators are given by the divisors $D_{Y_{\widetilde{\lambda^1}}}$ and $D_{Y_{\widetilde{\lambda^2}}}$, which will be defined later in Lemma \ref{lem3.7}. In other words,

\begin{introthm}\label{TheoremB}
The Nef cone of the Hilbert scheme $\mathrm{Hilb}_{P_{d,r}(t)}(OG(q;V))$ is a cone generated by $6$ classes by the pull-backs of generators of the Nef cones of two-step isotropic flag varieties and classes of divisors $D_{Y_{\widetilde{\lambda^1}}}$ and $D_{Y_{\widetilde{\lambda^2}}}$.
\end{introthm}

The explicit generators for the nef cone can be found in Theorem \ref{main}. For this theorem, we construct irreducible curves in the Hilbert scheme that play the role of dual classes to the generators of the N\'eron-Severi group of the Hilbert scheme. These curves arise from incidence relations derived from Schubert varieties in orthogonal Grassmannians.  

The paper is organized as follows. In \S\ref{sec2}, we provide necessary background on Hilbert schemes, Nef cones, and orthogonal Grassmannians. In \S\ref{sec3}, we discuss the geometry of the Hilbert scheme and present our first main result on its connectedness. In the last section \S\ref{sec4}, we compute the generators of the Nef cone of the Hilbert scheme as our main result.

\section{Preliminary}\label{sec2}
In this section, we review some facts on the Hilbert Schemes, N\'eron-Severi group and Grassmannians that will be used later in the paper.

Let $G$ be a connected, simply-connected, semisimple algebraic group $G$ of rank $n$ over complex numbers $\mathbb{C}$ and fix a maximal torus $T$ and Borel subgroup $B\subset G$ containing $T$. Here we take the set of roots in $B$ to be positive $R^+$. We also let $B^-$ be the opposite Borel subgroup corresponding to the negative roots. Let $\Delta=\{\alpha_1,\ldots,\alpha_n\}\subset\mathfrak{h}^*$ denote the set of simple roots in $R^+$. Let $W=N_G(T)/T$ be the corresponding Weyl group, which is generated by the simple reflections $\{s_1,\ldots,s_n\}$ where $s_i=s_{\alpha_i}$. 

There is a bijection between parabolic subgroups $P \subseteq G$ containing $B$ and subsets $\Delta_P \subseteq \Delta$, where $\Delta_P$ is the set of simple roots generating the Levi subgroup of $P$. Let $W_P \subseteq W$ be the Weyl group generated by the simple reflections in $\Delta_P$, and let $W^P$ denote the set of minimal length representatives for the quotient $W/W_P$. For any $w \in W^P$, we define the Schubert variety $\Omega_{w,P}$ as the closure of the corresponding $B^-$-orbit in $G/P$:
\[\Omega_{w,P} = \overline{B^- w P}/P.
\]
This variety has codimension $\ell(w)$, and we denote its fundamental class by 
\[
[\Omega_{w,P}] \in A^{\ell(w)}(G/P).\]

Given an increasing sequence $\mathbf{d}=(0<d_1<d_2<\cdots<d_s\leq n)$ of integers of length $1\leq s\leq n$, we denote by $P_{\mathbf{d}}$ the parabolic subgroup associated with the set of Levi roots $\Delta_P=\Delta\backslash \{\alpha_{d_1},\alpha_{d_2},\ldots,\alpha_{d_s}\}$, and let $G/P_{\mathbf{d}}$ be the corresponding partial flag variety. Let $P_{k}\subset G$ be the maximal parabolic subgroup associated with the omission of the simple root $\alpha_k$. Here the simple roots are indexed as in \cite{Humphreys}.

\subsection{Hilbert schemes}
We fix integers $r\geq 2, d\geq 3$ and $\mathcal{N}$ such that $ r\leq \mathcal{N}$.
Let $X$ be a degree $d$ hypersurface in $\mathbb{P}^r$. Recall the Hilbert polynomial 
$
P_{d,r}(t):=\binom{t+r}{r}-\binom{t+r-d}{r}
$
of $X$.

It is known from \cite[Theorem 3.1]{Seong} that for arbitrary subvarieties $X$ in $\mathbb{P}^{\mathcal{N}}$, the space of degree $d$ hypersurfaces of $\mathbb{P}^r$ in $X$ agrees with the Hilbert scheme $\mathrm{Hilb}_{P_{d,r}(T)}(X)$. Hence, we have the following statement.

\begin{thm}
There is a one-to-one correspondence between the space of degree $d$ hypersurfaces of $\mathbb{P}^r$ in $G/P_{d_i}$ and the Hilbert scheme $\mathrm{Hilb}_{P_{d,r}(t)}(G/P_{d_i})$. 
\end{thm}

With the identification of the Hilbert scheme $\mathrm{Hilb}_{P_{d,r}(t)}(G/P_{d_i})$ with the space of degree $d$ hypersurfaces of $\mathbb{P}^r$ in $G/P_{d_i}$, we will investigate $\mathrm{Hilb}_{P_{d,r}(t)}(G/P_{d_i})$ for odd and even orthogonal Grassmannians in the later section \ref{sec3}.

\subsection{Nef cones}

The following statement appears in \cite[\S 1.3]{KPZ}. We take this occasion to provide explicit Schubert divisor classes here, as this result will be used later.

\begin{thm}\label{thm2.3}
Given $\pi_i:G/P_{\mathbf{d}}\rightarrow G/P_{d_i}$ the projection for $1\leq i\leq s$, the N\'eron-Severi group of $G/P_{\mathbf{d}}$ is generated by the pull-backs of the classes $[\Omega_{s_i,P_{d_i}}]\in A^1(G/P_{d_i})$ of the unique Schubert divisor for $s_i\in W^{P_{d_i}}$ of $G/P_{d_i}$.
\end{thm}
\begin{proof}
Let $\{ \omega_{d_1},\ldots \omega_{d_s}\}$ be the set of fundamental weights corresponding to $\Delta\backslash \Delta_P$. For each maximal parabolic subgroup $P_{d_i}$, the Picard group $Pic(G/P_{d_i})\cong \mathbb{Z}$ is generated by the ample line bundle $\mathcal{L}_{\omega_{d_i}}$ on $G/P_{d_i}$ for $1\leq i\leq s$ \cite[P. 17]{BL}. Explicitly, for each $1\leq i\leq s$, $V_{\omega_{d_i}}$ denote the irreducible representation of $G$ with highest weight $\omega_i$. One can define a map 
\[
\psi_i: G/P_{\mathbf{d}}\rightarrow \mathbb{P}(V_{\omega_{d_i}})
\]
by sending $gP_{\mathbf{d}}$ to $g\cdot[v_i]$ where $[v_i] \in \mathbb{P}(V_{\omega_{d_i}})$ is the line spanned by a highest weight vector $v_i$. Since the stabilizer of the highest weight line is exactly is $P_{d_i}(\supset P_{\mathbf{d}})$ (\cite[Corollary 3.6]{Gross}), the image of the map $\psi_i$ is the $G$-orbit of $[v_i]$ in $\mathbb{P}(V_{\omega_{d_i}})$, and thus $\psi_i$ factors as
\begin{equation}\label{eqn2.1}
G/P_{\mathbf{d}}\xrightarrow{f_i} G/P_{d_i}\overset{\iota_i}{\hookrightarrow} \mathbb{P}(V_{\omega_{d_i}}).
\end{equation}
The Picard group $\mathrm{Pic}(G/P_{d_i})$ is generated by the line bundle $\mathcal{L}(\omega_{d_i}) := \iota_i^*(\mathcal{O}(1))$. The first Chern class of this bundle $c_1(\mathcal{L}_{\omega_{d_i}})$ is given by the Schubert class $\sigma^1_{s_i, P_{d_i}} = [\Omega_{s_i, P_{d_i}}]$ in the Chow group $A^1(G/P_{d_i})$. So, the identification $c_1(\mathcal{L}(\omega_{d_i})) = \sigma^1_{s_i, P_{d_i}}$ yields the isomorphism $\mathrm{Pic}(G/P_{d_i}) \cong A^1(G/P_{d_i})$. The pullbacks $f_i^*(\sigma^1_{s_i,P_{d_i}})$ form a basis for the Néron-Severi group $NS(G/P_{\mathbf{d}})$ (c.f. \cite[Proposition 4.2]{LRY}).
\end{proof}

In particular, the {\it minimal embedding} $\iota_i$ implies that the $G/P_{d_i}$ can be embedded into projective space $\mathbb{P}^\mathcal{N}_{\mathbb{C}}$, where the dimension $\mathcal{N}$ is determined by the Lie type of the group $G$. For the special linear group $G=SL(n)$ of type A, this dimension is given by 
$
\mathcal{N}=\binom{n}{d_i}-1.
$
For the special orthogonal group $G=SO(N)$ representing type B and type D, the dimension is given by
\[
\mathcal{N}=\binom{N}{d_i}-1.
\]
This applies to odd values of $N=2n+1$ for $n \geq 2$ with the restriction $d_i < n$. It similarly applies to even values of $N=2n$ for $n \geq 3$ with the restriction $d_i < n-1$.
See \cite[Pages 300-303]{OV90} for each dimention.

The following proposition summarizes the relationship between the Picard number of a smooth variety and its associated projective bundle, as derived in {\cite[Theorem 5.9]{EH16}}.

\begin{thm}\label{thm2.4}
Let $E$ be a vector bundle on a smooth variety $X$, and let $\pi: \mathbb{P}(E) \to X$ denote its projectivization. Then
$$\rho(\mathbb{P}(E)) = \rho(X) + 1,$$
where $\rho(X) = \mathrm{rank}(\mathrm{NS}(X))$ denotes the Picard number of $X$.
\end{thm}

We can explicitly describe the Nef cone by exploiting the orthogonality relations between divisors and curves. As the next statement shows \cite[Lemma 2.3]{Seong}, the existence of a dual basis of curves is sufficient to identify the generators of the Nef cone.

\begin{lemma}\label{l:gen}
Let $X$ be a projective variety with Picard number $\rho(X) = n$. Suppose that the divisor classes $D_1, \ldots, D_n$ form a basis for the Néron-Severi group $\mathrm{NS}(X)$. If there exist irreducible curves $\gamma_1, \ldots, \gamma_n \subset X$ satisfying the intersection conditions
\[
D_i \cdot \gamma_j = \delta_{ij} \quad \text{for all } 1 \leq i, j \leq n,
\]
then the Nef cone of $X$ is generated by $D_1, \ldots, D_n$.
\end{lemma}

\subsection{Orthogonal Grassmannians}

Let $V$ be a $N$-dimensional vector space over complex numbers, equipped with a non-degenerate symmetric bilinear form $\beta(\cdot,\cdot)$ on it. Let $E^\perp$ denote the orthogonal complement of $E$ with respect to $\beta$, defined as
\[
E^\perp = \{v \in V \mid \beta(v, u) = 0, \forall u \in E\}.
\]
 A $\mathfrak{m}$-dimensional linear subspace $E\subset V$ is called {\it isotropic} if $E\subset E^\perp$. We denote by $OG(q;V)$ the set of isotropic subspaces of dimension $q$ in $V$. 
 
For $q\leq\lfloor{N/2}\rfloor$, the orthogonal Grassmannain $OG(q;V)$ can be identified with $G/P_{q}$ where $G=SO(N)$ and $P_q$ is the maximal parabolic subgroup.

 Given a sequence 
 \[
 p_1<  p_2< \cdots<p_s\leq \lfloor N/2\rfloor
 \]
of integers, we consider a partial flag 
 \[
 \mathcal{F}_\bullet: 0= \mathcal{F}_{p_1}\subset  \mathcal{F}_{p_2}\subset \cdots\subset  \mathcal{F}_{p_s}\subset V
 \] 
 of isotropic subspaces of $V$ such that the dimensions of $ \mathcal{F}_{p_i}$ are $p_i$ for all $i$. The parabolic subgroup $P_I$ with $I=\{p_1,\ldots,p_s\}$ stabilizes the flag $ \mathcal{F}_\bullet$. We denote by $Fl^X(p_1,p_2,\ldots,p_s)=SO(N)/P_I$ the set of such isotropic flags where $X=B$ if $N$ is odd and $X=D$ if $N$ is even. And we call $Fl^B(p_1,p_2,\ldots,p_s)$ the flag variety of (Lie) type $B$ and $Fl^D(p_1,p_2,\ldots,p_s)$ of type $D$.

Let
$
\mathbf{k}: k_1<k_2<\cdots<k_s\;\text{and}\; \mathbf{p}:p_1<  p_2<\cdots<p_s\leq \lfloor N/2\rfloor 
$
be strictly increasing sequences satisfying
$
k_i-k_{i-1}\leq p_i-p_{i-1}.
$
We consider a fixed isotropic flag 
\[
F_{p_1}\subset F_{p_2}\subset \cdots \subset F_{p_s}\subset V
\]
whose dimensions are parametrized by $\mathbf{p}$.
Inside $OG(q:V)$, we can define the Schubert variety associated with a partition $\lambda:=\lambda(\mathbf{k},\mathbf{p},q)$ by
\[
\Omega_\lambda=\{\Sigma\mid \mathrm{dim}(\Sigma\cap F_{q_i})\geq k_i \}\subset OG(q;V).
\]
Here, the partition $\lambda:=\lambda(\mathbf{k},\mathbf{p},q)=(\lambda_1>\lambda_2>\cdots>\lambda_s)$ is defined by 
\[
\lambda_{k_i}=N-q-p_i
\]
and the rest of $\lambda$ is set by choosing the parts $\lambda_k$ minimally. The codimension of $\Omega_\lambda$ is given by $|\lambda|=\lambda_1+\cdots+\lambda_s$. We denote by 
\[
\sigma_\lambda:=[\Omega_\lambda]
\]
the cohomology class of the Schubert variety $\Omega_\lambda$.

Following \cite{AF}, there exists an element $w := w(\mathbf{k}, \mathbf{p},q) \in W^{P_q}$ corresponding to $\mathbf{k}$, $\mathbf{p}$, and $q$. As a result, $\Omega_{\lambda}$ coincides with the previously discussed Schubert variety $\Omega_{w,P_q}$.

\section{The geometry of the Hilbert schemes}\label{sec3}

In this section, we establish analogues of the connectedness results for the Hilbert scheme of orthogonal Grassmannians $OG(q; V)$, where $\dim V = N$, for general values of $q$ and $N$, paralleling \cite[\S 3]{Seong20} on ordinary Grassmannians, aside from certain statements on projective spaces and accompanying remarks. 

We begin with the following lemma describing the set of degree $d\geq 3$ hypersurfaces in a projective space of dimension $r\geq 1$ contained in an orthogonal Grassmannian. 
This lemma is necessary to compute the number of connected components of $\mathrm{Hilb}_{P_{d,r}(t)}(OG(q;V))$. 
Let $q<\lfloor N/2 \rfloor$, and set
\[
\mathcal{N}=\binom{N}{q}-1.
\]
Recall from \eqref{eqn2.1} that this represents the projective dimension associated with the minimal embedding $\iota_q:OG(q;V)=G/P_q\rightarrow \mathbb{P}^\mathcal{N}$.

\begin{lemma}\label{lem3.1}
Let $X$ be a degree $d$ hypersurface of $\mathbb{P}^r$ in $OG(q;V)$. Let $L$ be an $r$-dimensional projective space containing $X$ in $\mathbb{P}^\mathcal{N}$. 
Then $L$ is contained in $OG(q;V)$.
\end{lemma}
\begin{proof}

The orthogonal Grassmannian $OG(q;V)$ is cut out in $\mathbb{P}(\wedge^qV)$ by the vanishing of the quadratic hypersurfaces $Q_i$ derived from the Pl\"ucker relations and other specific quadratic equations $\mathcal{K}_i$. That is, 
\[
 OG(q;V)=(\cap Q_i) \cap \left(\cap_i\mathcal{K}_i\right)=Gr(q,2n)\cap \left(\cap_i\mathcal{K}_i\right).
\]
Detailed expressions for $Q_i$ and $\mathcal{K}_i$ can be found in \cite[Theorem 2.1]{EM}.
 
Since $X$ lies in both $OG(q;V)$ and $L$, it must be contained in the intersection $(\bigcap Q_i) \cap L \cap \left(\cap_i\mathcal{K}_i\right)$. By the proof of \cite[Theorem 3.2]{Seong}, we know that $L \subset Q_i$ for all $i$, so the containment simplifies to $X \subset L \cap \left(\cap_i\mathcal{K}_i\right)$. In fact, we can also conclude that $L\cap \mathcal{K}_i=L$ for all $i$ by the exact same argument as in \cite[Theorem 3.2]{Seong}. That is, for any $i$, if $L$ is not fully contained in $\mathcal{K}_i$, the intersection $L \cap \mathcal{K}_i$ would be a quadric inside $L$. This leads to a contradiction, as a hypersurface $X$ of degree $d \geq 3$ cannot be contained in the quadratic hypersurface of $L$. Hence $L\cap \mathcal{K}_i=L$, which proves the claim.
\end{proof}

We now ready to get the number of connected components in $\mathrm{Hilb}_{P_{d,r}(t)}(OG(q;V))$, as follows. Our standing assumption for the rest of this manuscript is 
\begin{equation}\label{e:assumption}
1<q< \left\lfloor \dfrac{N-1}{2}\right\rfloor , \;d\geq 3, \text{ and } 2<r\leq  \left\lfloor \dfrac{N-1}{2}\right\rfloor -q.
\end{equation}

\begin{thm}\label{thm3.2}
 Let $m=\mathrm{min}\{\lfloor (N-1)/2 \rfloor-q,q-1\}$.
If $2<r\leq m$, there are $2$ connected components in $\mathrm{Hilb}_{P_{d,r}(t)}(OG(q;V))$. 
In particular, if $m=q-1$, for $q-1<r\leq \lfloor (N-1)/2 \rfloor-q$, there is $1$ connected component in $\mathrm{Hilb}_{P_{d,r}(t)}(OG(q;V))$. 
\end{thm}
\begin{proof}
By Lemma \ref{lem3.1}, any projective spaces $L$ of dimension $r$ in $\mathbb{P}^\mathcal{N}$ containing a degree $d$ hypersurface $X$ in $\mathbb{P}^r\subset OG(q;V)$ have their cohomology classes $[L]$ in terms of the Schubert basis of $A^*(OG(q;V))$. The condition $r\leq \lfloor (N-1)/2 \rfloor-q$ is necessary here. If $r>\lfloor (N-1)/2 \rfloor-q$, $L$ is no longer of degree $1$ \cite[\S 2.3, \S 3.3]{BKT2009}, which contradicts that $L$ is a projective space.

Using the Pieri's formula \cite[Theorem 2.1, Theorem 3.1]{BKT2009}, the possible cohomology classes for $[L]$ are just either $\sigma_{\lambda^1}$ or $\sigma_{\lambda^2}$ where the partitions $\lambda^i$ for $i=1,2$ are given by
\begin{equation}\label{eqn:L}
\begin{aligned}
\lambda^1&=(N-q-1,N-q-2,\ldots,N-2q+2, N-2q+1,N-2q-r),\;\;\text{and}\\
\lambda^2&=(N-q-1,N-q-2,\ldots, N-2q+r,N-2q+r-2,\ldots, N-2q,N-2q-1)
\end{aligned}
\end{equation}
of length $q$. Applying the Pieri's formula again, the cohomology class $[X]$ of $X$ is given by
 \[
[X]=\begin{cases}
 d\cdot\sigma_{\mu^1}&\text{if }[L]=\sigma_{\lambda^1}\\
d\cdot\sigma_{\mu^2}&\text{if }[L]=\sigma_{\lambda^2}
 \end{cases},
 \]
 where the partitions $\mu^i$ for $i=1,2$ are given by
 \begin{align*}
 \mu^1&=(N-q-1,N-q-2,\ldots,N-2q+2, N-2q+1,N-2q-r+1),\;\;\text{and}\\
 \mu^2&=(N-q-1,N-q-2,\ldots, N-2q+r-1,N-2q+r-3,\ldots, N-2q,N-2q-1).
 \end{align*}
We further note that the corresponding dimension of $L$ is $r$ and its range is $0\leq r\leq N-2q$ if $[L]=\sigma_{\lambda^1}$, and $0\leq r\leq q-1$ if $[L]=\sigma_{\lambda^2}$.

According to \cite[Theorem 1.1]{HM}, projective spaces in the orthogonal Grassmannian with the same cohomology class as $L$ are related via linear automorphisms on the orthogonal Grassmannian. This implies that for each $i=1,2$, the space of $r$-dimensional projective spaces of cohomology class $\sigma_{\lambda^i}$ in $OG(q;V)$ is connected. As the space of all dgree $d$ hypersurfaces in $\mathbb{P}^r$ is also connected \cite{Hart}, the cohomology class $d\cdot\sigma_{\mu^i}$ deduces precisely one connected component for each $i=1,2$.

Given the assumptions \eqref{e:assumption}, we can conclude that if $2<r\leq m$, the Hilbert scheme $\mathrm{Hilb}_{P_{d,r}(r)}(OG(q;V))$ has two distinct connected components. In addition, if $m=q-1$ so that $q-1<r\leq \lfloor (N-1)/{2}\rfloor -q$, then the only possible $[L]$ is $\sigma_{\lambda^1}$. Hence, in this case, $\mathrm{Hilb}_{P_{d,r}(t)}(OG(q;V))$ has only one connected component. So, we get the desired results.
\end{proof}

We remark that when $\lfloor (N-1)/2 \rfloor-q=q-1$, the possible number of connected components in $\mathrm{Hilb}_{P_{d,r}(t)}(OG(q;V))$ is either exactly two or zero without the case that only one connected component exists in $\mathrm{Hilb}_{P_{d,r}(t)}(OG(q;V))$. Additionally, if any elements in each of two connected components have different cohomology classes, then these two compnents are disconnected, \cite[p. 3443]{Seong}. Thus we obtain the following corollary.
\begin{cor}
Suppose $d\geq 3$, $2<r\leq \mathrm{min}\{\lfloor(N-1)/2\rfloor-q,q-1\}$. The Hilbert scheme $\mathrm{Hilb}_{P_{d,r}(t)}(OG(q;V))$ is disconnected.
\end{cor}

\begin{rmk}
The arguments used to prove Theorem \ref{thm3.2} do not apply to isotropic Grassmannians, since the analogous projective spaces in the isotropic Grassmannians belong to type C where (co)homological rigidity is known to fail, see \cite[Theorem 1.1]{HM}.
\end{rmk}

 Let $1\leq p_1<p_2\leq \lfloor (N-1)/{2}\rfloor$. We consider the partial flag variety $Fl^X(p_1,p_2;V)$ of type $X=B$ or $D$ parametrizing flags
\[
\{\mathbf{0}\}\subset E_{p_1}\subset E_{p_2}\subset V\cong \mathbb{C}^{N},
\] 
of isotropic subspaces of $V$ where the subscription denote its dimension, i.e., $\mathrm{dim}(E_{p_i})=p_i$.

Let $\mathcal{U}_{p_i}$ denote the tautological isotropic subbundle of rank $p_i$ over $Fl^X(p_1, p_2; V)$ for $i=1,2$. We define the \textit{tautological quotient bundle} $\mathcal{Q}$ as the vector bundle $\mathcal{Q} \coloneqq \mathcal{U}_{p_2}/\mathcal{U}_{p_1}$ of rank $r+1$. Then we have the following identification of the Hilbert scheme $\mathrm{Hilb}_{P_{d,r}(t)}(OG(q;V))$ with the degree $d$ component of the symmetric algebra of the dual of the tautological bundle.

\begin{thm}\label{thm3.5}
Let $d\geq 3$, $2<r\leq \mathrm{min}\{\lfloor(N-1)/2\rfloor-q,q-1\}$. Then $\mathrm{Hilb}_{P_{d,r}(t)}(OG(q;V))$ is isomorphic to $\mathcal{M}=\mathbb{P}(\mathrm{Sym}^d\mathcal{S}^\vee)$. Here $\mathcal{S}$ is the bundle over the disjoint union of flag varieties 
\[
Fl^X(q-1,q+r;V)\sqcup Fl^X(q-r,q+1;V),
\]
defined by \[
\mathcal{S} := 
\begin{cases} 
\mathcal{Q}_1 & \text{on } \mathcal{X}_1=Fl^X(q-1,q+r;V), \\
\mathcal{Q}_2^\vee & \text{on } \mathcal{X}_2=Fl^X(q-r,q+1;V)
\end{cases}
\]
where $\mathcal{Q}_i$ is the tautological quotient bundles on the flag variety $\mathcal{X}_i$ for $i=1,2$.
\end{thm}
\begin{proof}

We take $1\leq p_1< p_1+2\leq p_2\leq \lfloor(N-1)/2\rfloor$ such that $p_1=q-1$ or $p_2=q+1$. We consider the three-step flag variety 
$
Y:=Fl^X(p_1, q, p_2; V)$ parametrizing 
\[
E_{p_1} \subset E_q \subset E_{p_2}\subset V
\]
where $E_q$ is an isotropic subspace of $V$, $\mathrm{dim}(E_q)=q$. Let $\pi:Y\rightarrow Fl^X(p_1,p_2;V)$ be the projection map that simply drops the middle isotropic subspace $E_q$. 
Then the partial flag variety $Fl^X(p_1,p_2;V)$ parametrizes the family of subvarieties
\begin{equation}\label{eqn3.1}
Z = \pi^{-1}(y) = \{U \mid E_{p_1}\subset U\subset E_{p_2}\} \subset OG(q;V)
\end{equation}
defined as the fibers over points $y=(E_{p_1},E_{p_2})$ in the base $Fl^X(p_1,p_2;V)$.
 We further note that $Z$ is isomorphic to an ordinary Grassmannian of the quotient space $E_{p_2}/E_{p_1}$ so that
\begin{align*}
Z &\cong \{ \bar{U} \in Gr(q-p_1, E_{p_2}/E_{p_1}) \} \\
  &\cong Gr(a, p_2-p_1), 
\end{align*}
viewed as a subvariety inside $OG(q;V)$, where $a=q-p_1$. 
Since $p_1=q-1$ or $p_2=q+1$, the fiber $Z$ is isomorphic to a projective space. Thus, the two-step flag variety parametrizes a family of linear spaces $\mathbb{P}^r$ inside $OG(q;V)$, where $r=p_2-p_1-1>1$. 
In particular, the cohomology class of $Z$ in the Chow ring $A^*(OG(q;V))$ of $OG(q;V)$ is given by
\begin{equation}\label{eqn3.2}
[Z]=\begin{cases}
\sigma_{\lambda^1} & \text{if } p_1=q-1, \\
\sigma_{\lambda^2} & \text{if } p_2=q+1,
\end{cases}
\end{equation}
via the incidence condition in \eqref{eqn3.1} where the partitions $\lambda^i$ for $i=1,2$ are given by
\begin{align*}
\lambda^1&=(N-q-1,N-q-2,\ldots,N-2q+2, \underbrace{N-2q+1}_{p_1^{\text{th}}\text{-position}},N-2q-r),\;\;\text{and}\\
\lambda^2&=(N-q-1,N-q-2,\ldots, \underbrace{N-2q+r}_{p_1^{\text{th}}\text{-position}},N-2q+r-2,\ldots, N-2q,\underbrace{N-2q-1}_{q^{\text{th}}\text{-position}}).
\end{align*}
Indeed, any projective spaces $\mathbb{P}^r$ in $OG(q;V)$ arise in this way. Specifically, the Schubert conditions 
\[
\mathrm{dim}(U_q\cap E_{p_1})\geq p_1\;\text{and }\mathrm{dim}(U_q\cap E_{p_2})\geq q
\]
for $\sigma_{\lambda^i}$ in \eqref{eqn:L} are identical to those specified in \eqref{eqn3.1}, including the partitions $\lambda^i$ for $i=1,2$. We note that for $N$ even, the above conditions imply the closure of the locus where equality holds, see \cite[\S 6]{FP}.

  To describe this structure globally, and treat the two cases uniformly, we define the bundle $\mathcal{S}$ of rank $r+1$ as
\begin{equation}
\mathcal{S} = \begin{cases}
\mathcal{Q}_1 & \text{if } q=p_1+1, \\
\mathcal{Q}_2^\vee & \text{if } q=p_2-1.
\end{cases}
\end{equation}
with the tautological quotient bundles $\mathcal{Q}_i$ for $i=1,2$.
With this definition, the three-step flag variety $Y$ is identified globally as the projectivization of $\mathcal{S}$, yielding the isomorphism $$Y \cong \mathbb{P}(\mathcal{S}).$$
  
  We then construct the bundle $\mathcal{M} \coloneqq \mathbb{P}(\mathrm{Sym}^d\mathcal{S}^\vee)$ over the base $Fl^X(p_1, p_2; V)$, $d\geq 3$. The fiber of $\mathcal{M}$ over a point $y\in Fl^X(p_1, p_2; V)$ parametrizes the set of degree $d$ hypersurfaces inside the fiber $Z \cong \mathbb{P}^r \subset OG(q;V)$.
Consequently the connected component in $\mathrm{Hilb}_{P_{d,r}(t)}(OG(q;V))$ corresponding to the class $d\cdot \sigma_{\mu^i}$ is isomorphic to the bundle $\mathcal{M}\rightarrow Fl^X(p_1, p_2; V)$ for each case, as desired.
 \end{proof}
 
 The following corollary characterizes the case where $\mathrm{Hilb}_{P_{d,r}(t)}(OG(q;V))$ consists of a single connected component:

\begin{cor}
Let $d\geq 3$, $r> 2$. If $q-1< r\leq \lfloor(N-1)/2\rfloor-q$, then 
$\mathrm{Hilb}_{P_{d,r}(t)}(OG(q;V))$ is isomorphic to $\mathcal{M}=\mathbb{P}(\mathrm{Sym}^d\mathcal{Q}_1^\vee)$,
where $\mathcal{Q}_1$ is the tautological quotient bundles on the flag variety $\mathcal{X}_1 =Fl^X(q-1,q+r;V)$.
\end{cor}

\section{The Nef cones of the Hilbert scheme}\label{sec4}
In this section, we explicitly construct the generators of the Nef cone of the Hilbert scheme $\mathrm{Hilb}_{P_{d,r}(t)}(OG(q;V))$ of the orthogonal Grassmannians for $1<q<\lfloor(N-1)/2\rfloor$ and $d\geq 3$, $2<r\leq \mathrm{min}\{\lfloor(N-1)/2\rfloor-q,q-1\}$. We accomplish this by using 1-dimensional families of hypersurfaces dual to the generators of the N\'eron-Severy group.

As stated in Theorem \ref{thm3.5}, $\mathrm{Hilb}_{P_{d,r}(t)}(OG(q;V))$ can be expressed as the bundle $\mathbb{P}(\mathrm{Sym}^d\mathcal{S}^\vee)$ over a disjoint union of two partial flag varieties $Fl^X(q-1,q+r;V)$ and $Fl^X(q-r,q+1;V)$ of type $X=B$ or $D$. We denote by $\mathcal{M}_1$ the projective bundle over the flag variety $Fl^X(q-1,q+r;V)$ and $\mathcal{M}_2$ the projective bundle over the other flag variety $Fl^X(q-r,q+1;V)$. To achieve our goal of this section, we first compute the Nef cones of $\mathcal{M}_1$ and $\mathcal{M}_2$ respectively, as the Nef cone of $\mathrm{Hilb}_{P_{d,r}(t)}(OG(q;V))$ is spanned by these cones $\mathrm{Nef}(\mathcal{M}_1)$ and $\mathrm{Nef}(\mathcal{M}_2)$. For notational simplicity, we write $\mathcal{X}_1$ and $\mathcal{X}_2$ for $Fl^X(q-1,q+r;V)$ and $Fl^X(q-r,q+1;V)$ respectively throughout this section.

As we have seen in the previous section, the Picard number of $\mathcal{M}_i$ is $3$ for $i=1,2$ via Theorem \ref{thm2.4}, and the two generators of the N\'eron-Severi group $NS(\mathcal{M}_i)$ are induced by the generators of $NS(\mathcal{X}_i)$ via Theorem \ref{thm2.3}. In other words, 
given the flag variety $\mathcal{X}_i=Fl^X(p_1^i,p_2^i;V)$ for $i=1,2$ where $(p_1^1,p_2^1)=(q-1,q+r)$ and $(p_1^2,p_2^2)=(q-r,q+1)$, we have projections to orthogonal Grassmannians $\pi_1^i:\mathcal{X}_i\rightarrow OG(p_1^i;V)$ and $\pi_2^i:\mathcal{X}_i\rightarrow OG(p_2^i;V)$.
\[
\begin{tikzcd}
& \mathcal{X}_i \arrow[dl,"\pi_1^i"'] \arrow[dr,"\pi_2^i"] & \\
OG(p_1^i;V)  & & OG(p_2^i;V)
\end{tikzcd}
\]
For each $i\in \{1,2\}$ and fixed complete flags $F^{(1,i)}_\bullet$ and $F^{(2,i)}_\bullet$, we define Schubert varieties $\Omega_1^{(1,i)}\subset OG(p_1^i;V)$ and $\Omega_1^{(2,i)}\subset OG(p_2^i;V)$ to be
\begin{equation}\label{e:Schubert condition1}
\Omega_1^{(1,i)}=\{E_{p_1^i}\mid\mathrm{dim}(E_{p_1^i}\cap F^{(1,i)}_{N-p_1^i})\geq 1\}\subset OG(p_1^i;V)
\end{equation}
and 
\begin{equation}\label{e:Schubert condition2}
\Omega_1^{(2,i)}=\{E_{p_2^i}\mid\mathrm{dim}(E_{p_2^i}\cap F^{(2,i)}_{N-p_2^i})\geq 1\}\subset OG(p_2^i;V)
\end{equation}
of codimensions $1$. Especially, for $N$ even, the special Schubert varieties are meant to be the closure of the locus where equality holds, see \cite[\S 6]{FP}. 
 We let $\sigma_1^{(j,i)}=[\Omega_1^{(j,i)}]$ be the first Schubert classes for $j=1,2$. Then by Theorem \ref{thm2.3}, the N\'eron-severi group of $\mathcal{X}_i$ is generated by $\sigma_1^{(1,i)}$ and $\sigma_1^{(2,i)}$, i.e.,
\[
NS(\mathcal{X}_i)=\mathbb{Z}(\pi_1^i)^*\sigma_1^{(1,i)}\oplus \mathbb{Z}(\pi_2^i)^*\sigma_1^{(2,i)}.
\]

To complete our description of $NS(\mathcal{M}_i)$, we must identify its final generator. The subsequent lemma introduces a variety which will later turn out to represent this remaining generator of the N\'eron-severi group of $\mathcal{M}_i$.

\begin{lemma}\label{lem3.7}
Let $Y_{\widetilde{\lambda^i}}\subseteq OG(q;V)$ be a special Schubert variety associated to $\widetilde{\lambda^i}$ for $i=1,2$, where
\[
\widetilde{\lambda^1}=(r) \;\;\text{and }\;\widetilde{\lambda^2}=(1^r):=(\underbrace{1,\ldots,1}_{r {\text{ times}}}).
\] 
Then the set 
\[
D_{Y_{\widetilde{\lambda^i}}}:=\{X\in \mathcal{M}_i \mid X\cap Y_{\widetilde{\lambda^i}}\neq \emptyset\}
\] 
is a subvariety of codimension $1$ in $\mathcal{M}_i$.
\end{lemma}
\begin{proof}

We first observe that a projective space $L\cong \mathbb{P}^r$ is uniquely determined by any contained degree $d\geq 3$ hypersurface $X$ of dimension $r-1$. Indeed, if there is another projective space $M\neq L$ of dimension $r$ containing $X$, then $X$ must be contained in $L\cap M$ which is a linear subspace of dimension at most $r-1$. However, the hypersurface $X$ of degree greater than $3$ cannot be in any linear subspaces of $L\cap M$. So, we must have $L=M$. 

So, we can define a map 
\[
\phi:\mathcal{M}_i\rightarrow \mathcal{M}_i\times \mathcal{X}_i, \;X\mapsto (X,L_X)
\]
where $L_X$ is the unique projective space of dimension $r$ containing the degree $d\geq 3$ hypersurface $X$, and further, by Lemma \ref{lem3.1}, $L_X$ is contained in $OG(q;V)$. In particular, we identified the flag variety $\mathcal{X}_i$ with the space of fibers $L\cong \pi^{-1}(y)$ that are projective spaces of dimension $r$ for some $y\in \mathcal{X}_i$ as in \eqref{eqn3.1}. Then the map $\phi$ induces an isomorphism between $\mathcal{M}_i$ and its image $\phi(\mathcal{M}_i)$.

We consider the image $\phi(D_{Y_{\widetilde{\lambda^i}}})$ as 
 \begin{equation}\label{eqn3.5}
 \mathcal{Y}=\{(X,L_X)\mid X\cap Y_{\widetilde{\lambda^i}}\neq \emptyset, X\subset L_X\subset OG(q;V) \}\subset \mathcal{M}_i\times \mathcal{X}_i
\end{equation}
Because of this isomorphism, it suffices to calculate the codimension of $\mathcal{Y}$ in $\phi(\mathcal{M}_i)$ rather than the codimension of $D_{Y_{\widetilde{\lambda^i}}}$ in $\mathcal{M}_i$. Since $\phi$ is isomorphism onto its image, the dimension of $\phi(\mathcal{M}_i)$ is the same as the dimension of $\mathcal{M}_i$, i.e., 
\[
\mathrm{dim}(\phi(\mathcal{M}_i))=\binom{r+d}{d}-1.
\] 
It remains for us to determine the dimension of $\mathcal{Y}=\phi(D_{Y_{\widetilde{\lambda^i}}})$.

The Pieri formula \cite[Theorem 2.1]{BKT2009} implies that $L_X$ and $Y_{\widetilde{\lambda^i}}$ intersect in one point in general position. For other special linear subspaces $L_X^0$, we find that the intersection $L_X^0\cap Y_{\widetilde{\lambda^i}}$ is a linear space of dimension at least one.
We therefore consider two cases separately: the general case $L_X$ and the special case $L_X^0$. 

So, we decompose $\mathcal{Y}$ as 
\[
\mathcal{Y}=\mathcal{Y}_{gen}\sqcup\mathcal{Y}_{sp}
\] 
where $\mathcal{Y}_{gen}$ is given by the same conditions as \eqref{eqn3.5} except that $L_X$ are general linear subspaces and $\mathcal{Y}_{sp}$ are the ones where $L_X$ are special linear subspaces $L_X^0$. 

For general subspaces $L_X$, we let $L_X \cap Y_{\widetilde{\lambda^i}}=\{p\}$ for some point $p$. Then on $L_X$, ${\mathcal{Y}}_{gen}$ is isomorphic to the space
\[
\{\widetilde{X} \mid p\in \widetilde{X}\subset L_X\cong \mathbb{P}^r\}\subset \mathcal{M}_i
\]
of dimension
\[
\binom{r+d}{d}-2,
\]
which is cut out by one equation corresponding to the incidence condition $p\in X$. 
In addition the set of special subspaces $L_X^0$ satisfying the condition $\mathrm{dim}(L_X^0\cap  Y_{\widetilde{\lambda^i}})\geq 1$ is of codimension at least $1$ in the flag variety $\mathcal{X}_i$, as it is a proper closed subvariety defined by the upper semicontinuity of fiber dimension for the projection from the incidence variety 
\[
\Sigma= \{ (L_X, E) \mid  E  \in L_X\cap Y_{\widetilde{\lambda^i}} \}\subset \mathcal{X}_i\times OG(q;V)
\]
where $(U,W)\in \mathcal{X}_i$. The codimension of $\mathcal{Y}_{sp}$ in $\phi(\mathcal{M}_i)$ restricted in the special $L_X^0$ has codimension at least $1$. 

Hence, putting all together, we conclude that the codimension of $\mathcal{Y}$ in $\phi(\mathcal{M}_i)$ is $1$, as desired.
\end{proof}

We now claim that a family of curves $D_{Y_{\widetilde{\lambda^i}}}$ yields a divisor class that is independent of the generators of $NS(\mathcal{X}_i)$ induced by $\sigma_1^{(1,i)}$ and $\sigma_1^{(2,i)}$.
We denote by 
\[
\varphi_i: \mathcal{M}_i\rightarrow \mathcal{X}_i
\]
 the projective bundle over $\mathcal{X}_i$, and recall $\pi_j^i:\mathcal{X}_i\rightarrow OG(p_j^i;V)$ the projection for each $j=1,2$. 
\begin{lemma}\label{l:gam}
For each $i\in\{1,2\}$, there exists a $1$-dimensional family $\gamma_i$ of hypersurfaces in $\mathcal{M}_i$ such that
\[
(\pi_j^i \circ \varphi_i)^*\bigl(\sigma_i^{(j,i)}\bigr)\cdot \gamma_i = 0
\quad (j=1,2),
\qquad
\bigl[D_{Y_{\widetilde{\lambda^i}}}\bigr]\cdot \gamma_i = 1.
\]
\end{lemma}
\begin{proof}
We fix $i\in \{1,2\}$.

Recall that the two step flag variety $\mathcal{X}_i$ parametrizes a family of projective space $L\cong\mathbb{P}^r$ of dimension $r$ in $OG(q;V)$ such that
for a point $(E_{p_1^i},E_{p_2^i})\in \mathcal{X}_i$, we have 
\[
\{U\in OG(q;V)\mid E_{p_1^i}\subset U\subset E_{p_2^i}\}\cong \mathbb{P}^r.
\]
One can find a pair $(E_{p_1^i},E_{p_2^i})\in \mathcal{X}_i$ via generality such that 
\begin{equation}\label{e:intersection-empty}
E_{p_1^i}\cap F_{N-p_1^i}^{(1,i)}=0,\quad E_{p_2^i}\cap F_{N-p_2^i}^{(2,i)}=0, \quad\text{and}\quad \mathbb{P}(\mathfrak{S}^i)\cap Y_{\widetilde{\lambda^i}}=\{W\}\subset OG(q;V)
\end{equation}
where $W$ is a $q$-dimensional isotropic subspace of $V$ such that $E_{p_1^i}\subset W\subset E_{p_2^i}$ and $W\in Y_{\widetilde{\lambda^i}}$, and 
\[
\mathfrak{S}^i:= 
\begin{cases} 
E_{q+r}/E_{q-1}&\text{if } i=1, \\
\left(E_{q+1}/E_{q-r}\right)^\vee&\text{if } i=2.
\end{cases}
\] 

We use the local coordinates of $\mathbb{P}(\mathfrak{S}^i)$ to view $W$ as a point in $\mathbb{P}^r$. We may therefore fix nonzero homogeneous polynomials $f$ and $g$ of degree $d$ such that $f(W)=0$ and $g(W)\neq 0$. Let $Z(h)$ denote the zero locus of a polynomial $h$. For $[s:t]\in \mathbb{P}^1$, the polynomial $sf+tg$ has degree $d$, and so $Z(sf+tg)$ is a degree $d$ hypersurface in $\mathbb{P}(\mathfrak{S}^i)$. We then define the $1$-dimensional family $\gamma_i$ in $\mathcal{M}_i$ by
\[
\gamma_i = \left\{ Z(sf+tg)  \in \mathcal{M}_i \mid  [s:t] \in \mathbb{P}^1 \right\}.
\]
The construction of $\mathcal{M}_i$ along with the Schubert conditions in $\Omega_1^{(j,i)}$ for $j=1,2$ (see \eqref{e:Schubert condition1}, \eqref{e:Schubert condition2}) and the way of defining $\gamma_i\subset \mathbb{P}(\mathfrak{S}^i)$ with \eqref{e:intersection-empty} yield that
\[
(\pi_j^i\circ \varphi_i)^*\sigma_1^{(j,i)}\cdot\gamma_i=0 \quad\text{for }j=1,2.
\]
Furthermore, $\gamma_i\cap Y_{\widetilde{\lambda^i}}\subset \mathbb{P}(\mathfrak{S}^i)\cap Y_{\widetilde{\lambda^1}}=\{W\}$ that is a single degree $d$ hypersurface in $\mathbb{P}(\mathfrak{S}^i)$, so that $D_{Y_{\widetilde{\lambda^i}}}\cap \gamma_i=\{Z(f)=W\}$. This implies $\bigl[D_{Y_{\widetilde{\lambda^i}}}\bigr]\cdot \gamma_i = 1$.
\end{proof}

By the preceding lemma, the family $\gamma_i$ acts as the dual to the class of the divisor $D_{Y_{\widetilde{\lambda^i}}}$. This implies that the class $\bigl[D_{Y_{\widetilde{\lambda^i}}}\bigr]$ is linearly independent from the classes $(\pi_j^i\circ \varphi_i)^*\sigma_1^{(j,i)}$ for $j=1,2$. Hence, we obtain that the N\'eron-Severi group $NS(\mathcal{M}_i)$ is generated by pull-back of $NS(\mathcal{X}_i)$ and a divisor class associated to $D_{Y_{\widetilde{\lambda^i}}}$.

In the successive lemmas, we specify the irreducible curves needed to present the Nef cone of the Hilbert scheme.

\begin{lemma}\label{lem:gam1}
For each $i\in\{1,2\}$, there exists a $1$-dimensional family $\gamma'_i$ of hypersurfaces in $\mathcal{M}_i$ such that
\[
(\pi_1^i \circ \varphi_i)^*\bigl(\sigma_1^{(1,i)}\bigr)\cdot \gamma'_i = 1,\quad (\pi_2^i \circ \varphi_i)^*\bigl(\sigma_2^{(2,i)}\bigr)\cdot \gamma'_i = 0,\quad\text{and}\quad
\bigl[D_{Y_{\widetilde{\lambda^i}}}\bigr]\cdot \gamma'_i = 0.
\]
\end{lemma}
\begin{proof}
We begin with arguing uniformly in $i\in\{1,2\}$. Given fixed complete isotropic flags $F^{(j,i)}_\bullet$ for $j=1,2$, we first consider Schubert divisors
 \begin{align*}
 \Sigma^{(1,i)}&=(\pi_1^i)^{-1}(\Omega_1^{(1,i)})=\{E_{p_1^i}\subset E_{p_2^i}\mid \mathrm{dim}(E_{p_1^i}\cap F_{N-p_1^i}^{(1,i)})\geq 1\},\\
 \Sigma^{(2,i)}&=(\pi_2^i)^{-1}(\Omega_1^{(1,i)})=\{E_{p_1^i}\subset E_{p_2^i}\mid \mathrm{dim}(E_{p_2^i}\cap F_{N-p_2^i}^{(2,i)})\geq 1\}
 \end{align*}
of codimension $1$ in $\mathcal{X}_i$.  Here when $N$ is even, the above divisors should be read as the closure of the locus where equality holds, see \cite[\S 6]{FP}.

One can find a quadruple $(\widetilde{F}_{p_1^i}^{(2,i)},\widetilde{F}^{(1,i)}_{q^i},\widetilde{F}^{(2,i)}_{q^i},\widetilde{F}_{p_2^i}^{(1,i)})$ of isotropic spaces by generality, satisfying
\begin{equation}\label{e:vanishing}
\widetilde{F}_{p_1^i}^{(2,i)}\cap F_{N-p_1^i}^{(1,i)}=0, \widetilde{F}_{p_2^i}^{(1,i)}\cap F_{N-p_2^i}^{(2,i)}=0\text{ and } \widetilde{F}^{(j,i)}_{q^i}\notin Y_{\widetilde{\lambda^i}},
\end{equation}
for $j=1,2$ where $q=q^1=p_1^1+1= q^2=p_2^2-1$. Here the subscripts indicate their dimension, dim$(\widetilde{F}^{(j,i)}_k)=k$.  Then we take complete (or partial) flags $\widetilde{F}^{(j,i)}_\bullet$ of the isotropic subspaces of $V$ satisfying the conditions \eqref{e:vanishing}.
With the fixed flags $\widetilde{F}^{(j,i)}_\bullet$, let $\widetilde{\Sigma^{(1,i)}}$ denote the dual Schubert curve of $\Sigma^{(1,i)}$ given by 
\[
\widetilde{\Sigma^{(1,i)}}=\{\widetilde{E}_{p_1^i}\subset \widetilde{E}_{p_2^i}\mid \widetilde{F}^{(1,i)}_{p_1^i-1}\subset \widetilde{E}_{p_1^i}^1\subset \widetilde{F}^{(1,i)}_{p_1^i+1}, \;\widetilde{E}_{p_2^i}=\widetilde{F}^{(1,i)}_{p_2^i}\text{ fixed}\}\subset \mathcal{X}_i
\]
and $\widetilde{\Sigma^{(2,i)}}$ the dual Schubert curve of $\Sigma^{(2,i)}$ by 
\[
\widetilde{\Sigma^{(2,i)}}=\{\widetilde{E}_{p_1^i}\subset \widetilde{E}_{p_2^i}\mid \widetilde{F}^{(2,i)}_{p_2^i-1}\subset \widetilde{E}_{p_2^i}\subset \widetilde{F}^{(2,i)}_{p_2^i+1}, \;\widetilde{E}_{p_1^i}=\widetilde{F}^{(2,i)}_{p_1^i}\text{ fixed}\}\subset \mathcal{X}_i.
\]
 These Schubert curves are isomorphic to $\mathbb{P}^1$.

We construct the $1$-dimensional family $\gamma_i'$ of hypersurfaces in $\mathcal{M}_i$, as follows.
We define $\mathcal{S}_p^i$ for $p=(\widetilde{E}_{p_1^i},\widetilde{E}_{p_2^i})\in \widetilde{\Sigma^{(1,i)}}$ to be 
\begin{equation}\label{e:quotient}
\mathcal{S}_p^i := 
\begin{cases} 
\widetilde{E}_{p_2^i}/\widetilde{E}_{p_1^i}&\text{if } i=1, \\
\left(\widetilde{E}_{p_2^i}/\widetilde{E}_{p_1^i}\right)^\vee&\text{if } i=2.
\end{cases}
\end{equation}
For any point $p\in \widetilde{\Sigma^{(1,i)}}$, there is a natural surjective map
\[
\mathcal{Q}^1_{p}:\mathcal{S}_p^1=\widetilde{F}^{(1,1)}_{q+r}/\widetilde{E}_{q-1}\twoheadrightarrow \widetilde{F}^{(1,1)}_{q+r}/\widetilde{F}^{(1,1)}_{q}=: \mathcal{H}^1
\]
for $i=1$. Similarly, for $i=2$, we have a map
\[
\mathcal{Q}^2_{p}:\mathcal{S}_p^2=\left(\widetilde{F}^{(1,2)}_{q+1}/\widetilde{E}_{q-r}\right)^\vee\hookrightarrow \left(\widetilde{F}^{(1,2)}_{q+1}/\widetilde{F}^{(1,2)}_{q-r-1}\right)^\vee=:\mathcal{H}^2
\] from the exact sequence
\[
0\rightarrow  \widetilde{E}_{q-r}/\widetilde{F}^{(1,2)}_{q-r-1}\rightarrow \widetilde{F}^{(1,2)}_{q+1}/\widetilde{F}^{(1,2)}_{q-r-1}\rightarrow \widetilde{F}^{(1,2)}_{q+1}/\widetilde{E}_{q-r}\rightarrow 0.
\]
 Then these induce a map
\[
\hat{\mathcal{Q}}^i_p:\mathbb{P}\left(\mathrm{Sym}^d\left(\mathcal{H}^i\right)^\vee\right)\rightarrow \mathbb{P}\left(\mathrm{Sym}^d\left(\mathcal{S}_p^i\right)^\vee\right).
\] We note that this map is injective for $i=1$ and surjective for $i=2$.
We choose a general family of homogeneous polynomials $\Phi^i_p$ of degree $d$ for $p \in \widetilde{\Sigma^{(1,i)}}$ in $\mathbb{P}(\mathrm{Sym}^d(\mathcal{H}^i)^\vee)$. 
Then for each $p \in \widetilde{\Sigma^{(1,i)}}$, we define a hypersurface 
\[
X_{p} := Z(\Phi^i_p \circ \mathcal{Q}^i_p) 
\subset \mathbb{P}(\mathcal{S}_p^i).
\]
of degree $d$. This construction gives a morphism 
\begin{equation}\label{e:gamma1}
\widetilde{\Sigma^{(1,i)}} \longrightarrow \mathcal{M}_i, 
\quad p \longmapsto X_p.
\end{equation}
Since the domain of $\mathcal{Q}_p^i$ and the kernel of $\mathcal{Q}^1_p$, 
\[
\mathrm{ker}(\mathcal{Q}^1_p) = \widetilde{F}^{(1,1)}_{q}/\widetilde{E}_{q-1},
\]
vary with $p \in \widetilde{\Sigma^{(1,i)}}$, the point $\Phi^i_p \circ \mathcal Q^i_p$ in $\mathbb P(\mathrm{Sym}^d(\mathcal{S}_p^i)^\vee)$ varies with $p$. Hence, the morphism \eqref{e:gamma1} is non-constant, and its image
\[
\gamma'_i := \{\, X_{p} = Z(\Phi^i_p \circ \mathcal Q^i_{p}) \mid p \in \widetilde{\Sigma^{(1,i)}} \,\}
\subset \mathcal{M}_i
\]
is an irreducible curve. In particular, $\gamma'_i$ is generically isomorphic to $\widetilde{\Sigma^{(1,i)}} \cong \mathbb{P}^1$. It is worthwhile to note that by construction, $\varphi_i(X_p)=p$ for all $p\in \widetilde{\Sigma^{(1,i)}}$. 

With the family $\gamma'_i$, we have that
 $\gamma'_i\cdot (\pi_1^i\circ \varphi_i)^*\sigma_1^{(1,i)}=1$, as $\Sigma^{(1,i)}$ intersects with $\widetilde{\Sigma^{(1,i)}}$ at a single point, while $\gamma'_i\cdot (\pi_2^i\circ \varphi_i)^*\sigma_1^{(2,i)}$ vanishes by the conditions \eqref{e:vanishing}. Specifically, by \cite[Example 2.4.3]{Ful}, we obtain
\begin{align*}
\gamma'_i\cdot (\pi_j^i\circ \varphi_i)^*\sigma_1^{(j,i)}&=(\varphi_i)_*\gamma'_i\cdot (\pi_j^i)^*\sigma_1^{(j,i)}=\delta_{1j}
\end{align*}
where $\delta_{1j}=1$ if $j=1$ and $0$ if $j=2$.

Now we claim that $\gamma_i'\cdot [D_{Y_{\widetilde{\lambda^i}}}]$ vanishes. 

To compute $\gamma'_i \cdot [D_{Y_{\widetilde{\lambda^i}}}]$, we must determine if the moving point $\{W^i_p\} = \mathbb{P}(\mathcal{S}_p^i) \cap Y_{\widetilde{\lambda^i}}$ which varies depending on $\widetilde{E}_{p_1^i}$ is ever contained in the hypersurface $X_p$ for $p\in\widetilde{\Sigma^{(1,i)}}$. We remark that by the condition in \eqref{e:vanishing}, $W^1_p\notin \mathrm{Ker}(\mathcal{Q}_p^1).$ So, treating the point $W^i_p$ as its underlying line in $\mathcal{S}_p^i$, its image $\mathcal{Q}^i_p(W^i_p)$ is a single, constant $1$-dimensional subspace of $\mathcal{H}^i$. Since the image is a point in $\mathbb{P}(\mathcal{H}^i)$, a general choice of $\Phi^i_p$ ensures that it does not vanish at this point, implying $\Phi^i_p(\mathcal{Q}^i_p(W^i_p))\neq 0$. Consequently, the hypersurface $X_p$ does not contain the point $W^i_p$. It follows that the curve $\gamma'_i$ is disjoint from the divisor $Y_{\widetilde{\lambda^i}}$. 
\end{proof}

While the proof of Lemma \ref{lem:gam1} is inspired by the proof sketch provided by \cite[p.7]{Seong} for the ordinary Grassmannian case, we have supplied the explicit and concrete details required to complete the argument for orthogonal cases.

\begin{lemma}\label{lem:gam2}
For each $i\in\{1,2\}$, there exists a $1$-dimensional family $\gamma''_i$ of hypersurfaces in $\mathcal{M}_i$ such that
\[
(\pi_1^i \circ \varphi_i)^*\bigl(\sigma_1^{(1,i)}\bigr)\cdot \gamma''_i = 0,\quad (\pi_2^i \circ \varphi_i)^*\bigl(\sigma_2^{(2,i)}\bigr)\cdot \gamma''_i = 1,\quad\text{and}\quad
\bigl[D_{Y_{\widetilde{\lambda^i}}}\bigr]\cdot \gamma''_i = 0.
\]
\end{lemma}
\begin{proof}
Let $i\in \{1,2\}$. The proof is almost identical to that of Lemma \ref{lem:gam1}, except for the definition of $\gamma_i''$. More precisely, the Schubert divisors $\Sigma^{(j,i)}$ in $\mathcal{X}_i$ and the dual Schubert curves $\widetilde{\Sigma^{(j,i)}}$ ($j=1,2$), together with the quadruple $(\widetilde{F}_{p_1^i}^{(2,i)},\widetilde{F}_{q^i}^{(1,i)},\widetilde{F}_{q^i}^{(2,i)},\widetilde{F}_{p_2^i}^{(1,i)})$, are defined as in the proof of Lemma \ref{lem:gam1}.

To begin defining the $1$-dimensional family $\gamma''_i$ of hypersurfaces in $\mathcal{M}_i$, we let $\mathcal{S}_v^i$ be given by the expression in \eqref{e:quotient}, using $v=(\widetilde{E}_{p_1^i}, \widetilde{E}_{p_2^i}) \in \widetilde{\Sigma^{(2,i)}}$ in place of $p \in \widetilde{\Sigma^{(1,i)}}$. For each $v=(\widetilde{E}_{p_1^i}, \widetilde{E}_{p_2^i}) \in \widetilde{\Sigma^{(2,i)}}$, We define an injective map 
\[
\mathcal{Q}^1_{v}:\mathcal{S}_v^1=\widetilde{E}_{q+r}/\widetilde{F}^{(2,1)}_{q-1}\hookrightarrow \widetilde{F}^{(2,1)}_{q+r+1}/\widetilde{F}^{(2,1)}_{q-1}=: \mathcal{H}^1
\]
as the case for $i=1$. For $i=2$, we consider the surjective map
\[
\mathcal{Q}^2_{v }:\mathcal{S}_v^2=\left(\widetilde{E}_{q+1}/\widetilde{F}^{(2,2)}_{q-r}\right)^\vee\twoheadrightarrow\left(\widetilde{F}^{(2,2)}_{q}/\widetilde{F}^{(2,2)}_{q-r}\right)^\vee =:\mathcal{H}^2.
\] 
Then we have the induced map
\[
\hat{\mathcal{Q}}^i_v:\mathbb{P}\left(\mathrm{Sym}^d\left(\mathcal{H}^i\right)^\vee\right)\rightarrow \mathbb{P}\left(\mathrm{Sym}^d\left(\mathcal{S}_v^i\right)^\vee\right).
\] 
We take a general family of $\Phi_v^i$, homogeneous polynomials of degree $d$, for $v\in \widetilde{\Sigma^{(2,i)}}$ in $\mathbb P(\mathrm{Sym}^d(\mathcal{S}_v^i)^\vee)$, and define a hypersurfeace $X_v:=Z(\Phi_v^i\circ\mathcal{Q}_v^i)$ in $\mathbb{P}(\mathcal{S}^i_v)$ of degree $d$. In this way we obtain a morphism 
\begin{equation}\label{e:gamma2}
\widetilde{\Sigma^{(2,i)}} \longrightarrow \mathcal{M}_i, 
\quad v \longmapsto X_v.
\end{equation}

From the exact sequence
\[
0\rightarrow \widetilde{F}^{(2,2)}_{q}/\widetilde{F}^{(2,2)}_{q-r} \rightarrow \widetilde{E}_{q+1}/\widetilde{F}^{(2,2)}_{q-r}\rightarrow \widetilde{E}_{q+1}/ \widetilde{F}^{(2,2)}_{q}\rightarrow 0,
\]
we get the kernel
\[
\mathrm{ker}(\mathcal{Q}_v^2)=\left(\widetilde{E}_{q+1}/ \widetilde{F}^{(2,2)}_{q}\right)^\vee
\]
which varies with $v=(\widetilde{F}^{(2,2)}_{q-r},\widetilde{E}_{q+1})\in \widetilde{\Sigma^{(2,2)}}$. Moreover, the domains of $\mathcal{Q}_v^i$ depends on $v$, and so does $\Phi_v^i\circ\mathcal{Q}_v^i$. It follows that the morphism \eqref{e:gamma2} is not constant. We denote its image by
\[
\gamma''_i := \{\, X_{v} = Z(\Phi^i_v \circ \mathcal Q^i_{v}) \mid q\in \widetilde{\Sigma^{(2,i)}} \,\}
\subset \mathcal{M}_i.
\]
Then $\gamma''_i$ is an irreducible curve, as the image of $\mathbb{P}^1$ under a non-constant morphism. 

The fact that $\Sigma^{(2,i)}$ and $\widetilde{\Sigma^{(2,i)}}$ meets at a single point results in 
\[
(\pi_2^i \circ \varphi_i)^*\bigl(\sigma_2^{(2,i)}\bigr)\cdot \gamma''_i = 1
\] 
and the condition \eqref{e:vanishing} leads to 
\[
(\pi_1^i \circ \varphi_i)^*\bigl(\sigma_1^{(1,i)}\bigr)\cdot \gamma''_i = 0.
\]
Furthermore, the point $\{W^2_v\}=\mathbb{P}(\mathcal{S}_v^2)\cap Y_{\widetilde{\lambda^2}}$ does not belong to the kernel $\mathrm{Ker}(\mathcal{Q}_v^2)$, since $W^2_v$ and $\widetilde{F}_q^{(2,2)}$ are distinct by the condition \eqref{e:vanishing}. The rest of the argument follows line by line as in the proof of Lemma \ref{lem:gam1}.
\end{proof}

With the necessary groundwork established, we are now in a position to state the generators of the Nef cone of $\mathcal{M}_i$ for $i=1,2$. 

\begin{thm}\label{t:generators}
Let us fix $i\in\{1,2\}$. The Nef cone $Nef(\mathcal{M}_i)$ is spanned by the classes 
\[
(\pi_j^i\circ \varphi_i)^*\sigma_1^{(j,i)}\quad\text{ for }j=1,2\quad\text{ and }\bigl[D_{Y_{\widetilde{\lambda^i}}}\bigr].
\]
\end{thm}
\begin{proof}
We know from Lemma \ref{l:gam} that $NS(\mathcal{M}_i)$ is generated by $(\pi_1^i\circ \varphi_i)^*\sigma_1^{(1,i)}$, $(\pi_2^i\circ \varphi_i)^*\sigma_1^{(2,i)}$ and $\bigl[D_{Y_{\widetilde{\lambda^i}}}\bigr]$ for $i=1,2$. Also, by Lemmas \ref{l:gam}, \ref{lem:gam1}, and \ref{lem:gam2}, we have 
\begin{align*}
(\pi_1^i \circ \varphi_i)^*\bigl(\sigma_1^{(1,i)}\bigr)\cdot \gamma_i &= 0,&(\pi_2^i \circ \varphi_i)^*\bigl(\sigma_2^{(2,i)}\bigr)\cdot \gamma_i &= 0,&\bigl[D_{Y_{\widetilde{\lambda^i}}}\bigr]\cdot \gamma_i &= 1,\\
(\pi_1^i \circ \varphi_i)^*\bigl(\sigma_1^{(1,i)}\bigr)\cdot \gamma'_i &= 1,&(\pi_2^i \circ \varphi_i)^*\bigl(\sigma_2^{(2,i)}\bigr)\cdot \gamma'_i &= 0,&\bigl[D_{Y_{\widetilde{\lambda^i}}}\bigr]\cdot \gamma'_i &= 0,
\\
(\pi_1^i \circ \varphi_i)^*\bigl(\sigma_1^{(1,i)}\bigr)\cdot \gamma''_i &= 0,&(\pi_2^i \circ \varphi_i)^*\bigl(\sigma_2^{(2,i)}\bigr)\cdot \gamma''_i &= 1,&\bigl[D_{Y_{\widetilde{\lambda^i}}}\bigr]\cdot \gamma''_i &= 0.
\end{align*}
Hence, an application of Lemma \ref{l:gen} completes the proof.
\end{proof}

Finally, we arrive at our main result below on the generators of the Nef cone of the Hilbert scheme $\mathrm{Hilb}_{P_{d,r}(t)}(OG(q;V))$. 

\begin{thm}\label{main}
The nef cone of the Hilbert scheme $\mathrm{Hilb}_{P_{d,r}(t)}(OG(q;V))$ is a cone generated b $6$ classes
\begin{align*}
(\pi_1^1\circ \varphi_1)^*\sigma_1^{(1,1)},\;\;(\pi_2^1\circ \varphi_1)^*\sigma_1^{(2,1)},\;\;\bigl[D_{Y_{\widetilde{\lambda^1}}}\bigr],\;\; (\pi_1^2\circ \varphi_2)^*\sigma_1^{(1,2)},\;\; (\pi_2^2\circ \varphi_2)^*\sigma_1^{(2,2)}\;\;\text{and}\;\;\bigl[D_{Y_{\widetilde{\lambda^2}}}\bigr].
\end{align*}
\end{thm}
\begin{proof}
It follows by Theorem \ref{t:generators} and Theorem \ref{thm3.5}.
\end{proof}

\begin{ack}
We are grateful to Ignacio Sols for raising the question of the connectedness of the Hilbert scheme of the ordinary Grassmannians at the conference Algebraic Geometry and its Broader Implications, thereby bringing the problem to our attention through Seong’s work. We are indebted to Seong for his inspiring work on this question and for his computations of the nef cone, which have motivated our further investigation of analogous questions for the other types of Grassmannians.
The author started this work at the Institute for Basic Science and wishes to thank the members of the Center for Complex Geometry for their valuable support. 
\end{ack}
\bibliographystyle{alpha}
\bibliography{biblio}

\begin{bibdiv}
\begin{biblist}

\bib{AF}{article}{
      author={Anderson, David},
      author={Fulton, William},
       title={Chern class formulas for classical-type degeneracy loci},
        date={2018},
     journal={Compos. Math.},
      volume={154},
      number={8},
       pages={1746\ndash 1774},
}

\bib{BC13}{incollection}{
      author={Bertram, Aaron},
      author={Coskun, Izzet},
       title={The birational geometry of the {H}ilbert scheme of points on
  surfaces},
        date={2013},
   booktitle={Birational geometry, rational curves, and arithmetic},
      series={Simons Symp.},
   publisher={Springer, Cham},
       pages={15\ndash 55},
}

\bib{BL}{book}{
      author={Billey, Sara},
      author={Lakshmibai, V.},
       title={Singular loci of {S}chubert varieties},
      series={Progress in Mathematics},
   publisher={Birkh\"auser Boston, Inc., Boston, MA},
        date={2000},
      volume={182},
}

\bib{BP14}{article}{
      author={Biswas, Indranil},
      author={Parameswaran, A.~J.},
       title={Nef cone of flag bundles over a curve},
        date={2014},
     journal={Kyoto J. Math.},
      volume={54},
      number={2},
       pages={353\ndash 366},
}

\bib{BHLRSWZ}{article}{
      author={Bolognese, Barbara},
      author={Huizenga, Jack},
      author={Lin, Yinbang},
      author={Riedl, Eric},
      author={Schmidt, Benjamin},
      author={Woolf, Matthew},
      author={Zhao, Xiaolei},
       title={Nef cones of {H}ilbert schemes of points on surfaces},
        date={2016},
     journal={Algebra Number Theory},
      volume={10},
      number={4},
       pages={907\ndash 930},
}

\bib{BKT2009}{article}{
      author={Buch, Anders~Skovsted},
      author={Kresch, Andrew},
      author={Tamvakis, Harry},
       title={Quantum {P}ieri rules for isotropic {G}rassmannians},
        date={2009},
     journal={Invent. Math.},
      volume={178},
      number={2},
       pages={345\ndash 405},
}

\bib{LRY}{article}{
      author={Changzheng~Li, Mingzhi~Yang, Konstanze~Rietsch},
       title={An anticanonical perspective on g/p schubert varieties},
        date={2025},
     journal={preprint, arxiv:2506.18388},
}

\bib{EH16}{book}{
      author={Eisenbud, David},
      author={Harris, Joe},
       title={3264 and all that---a second course in algebraic geometry},
   publisher={Cambridge University Press, Cambridge},
        date={2016},
}

\bib{EM}{article}{
      author={El~Maazouz, Yassine},
      author={Mandelshtam, Yelena},
       title={The positive orthogonal {G}rassmannian},
        date={2025},
     journal={Matematiche (Catania)},
      volume={80},
      number={1},
       pages={279\ndash 302},
}

\bib{Ful}{book}{
      author={Fulton, William},
       title={Intersection theory},
     edition={Second},
      series={Ergebnisse der Mathematik und ihrer Grenzgebiete. 3. Folge. A
  Series of Modern Surveys in Mathematics [Results in Mathematics and Related
  Areas. 3rd Series. A Series of Modern Surveys in Mathematics]},
   publisher={Springer-Verlag, Berlin},
        date={1998},
      volume={2},
}

\bib{FP}{book}{
      author={Fulton, William},
      author={Pragacz, Piotr},
       title={Schubert varieties and degeneracy loci},
      series={Lecture Notes in Mathematics},
   publisher={Springer-Verlag, Berlin},
        date={1998},
      volume={1689},
        note={Appendix J by the authors in collaboration with I.
  Ciocan-Fontanine},
}

\bib{Gross}{book}{
      author={Grosshans, Frank~D.},
       title={Algebraic homogeneous spaces and invariant theory},
      series={Lecture Notes in Mathematics},
   publisher={Springer-Verlag, Berlin},
        date={1997},
      volume={1673},
}

\bib{Hart}{article}{
      author={Hartshorne, Robin},
       title={Connectedness of the {H}ilbert scheme},
        date={1966},
     journal={Inst. Hautes \'Etudes Sci. Publ. Math.},
      number={29},
       pages={5\ndash 48},
}

\bib{HM}{article}{
      author={Hong, Jaehyun},
      author={Mok, Ngaiming},
       title={Characterization of smooth {S}chubert varieties in rational
  homogeneous manifolds of {P}icard number 1},
        date={2013},
     journal={J. Algebraic Geom.},
      volume={22},
      number={2},
       pages={333\ndash 362},
}

\bib{Humphreys}{book}{
      author={Humphreys, James~E.},
       title={Introduction to {L}ie algebras and representation theory},
      series={Graduate Texts in Mathematics},
   publisher={Springer-Verlag, New York-Berlin},
        date={1978},
      volume={9},
        note={Second printing, revised},
}

\bib{KPZ}{incollection}{
      author={Kishimoto, Takashi},
      author={Prokhorov, Yuri},
      author={Zaidenberg, Mikhail},
       title={Group actions on affine cones},
        date={2011},
   booktitle={Affine algebraic geometry},
      series={CRM Proc. Lecture Notes},
      volume={54},
   publisher={Amer. Math. Soc., Providence, RI},
       pages={123\ndash 163},
}

\bib{Miyaoka}{incollection}{
      author={Miyaoka, Yoichi},
       title={The {C}hern classes and {K}odaira dimension of a minimal
  variety},
        date={1987},
   booktitle={Algebraic geometry, {S}endai, 1985},
      series={Adv. Stud. Pure Math.},
      volume={10},
   publisher={North-Holland, Amsterdam},
       pages={449\ndash 476},
}

\bib{OV90}{book}{
      author={Onishchik, A.~L.},
      author={Vinberg, \`E.\~B.},
       title={Lie groups and algebraic groups},
      series={Springer Series in Soviet Mathematics},
   publisher={Springer-Verlag, Berlin},
        date={1990},
        note={Translated from the Russian and with a preface by D. A. Leites},
}

\bib{Seong20}{article}{
      author={Seong, See-Hak},
       title={The {H}ilbert scheme of the {G}rassmannian is not connected},
        date={2020},
     journal={Comm. Algebra},
      volume={48},
      number={8},
       pages={3439\ndash 3446},
}

\bib{Seong}{article}{
      author={Seong, See-Hak},
       title={The nef cone of the hilbert scheme of hypersurfaces in the
  grassmannian},
        date={2020},
     journal={preprint, arxiv:2005.08266},
}

\end{biblist}
\end{bibdiv}
\end{document}